\documentclass[aop,MSNbibl,seceqn,dvips]{arximspdf}

\doi{10.1214/12-AOP785} 
\volume{41}
\issue{2}
\pubyear{2013}
\firstpage{848}
\lastpage{870}

\makeatletter

\newtheorem{theorem}{Theorem}[section]
\newtheorem{lemma}[theorem]{Lemma}

\newcommand{\R}{\mathbb R}
\newcommand{\Q}{\mathbb Q}
\newcommand{\N}{\mathbb N}
\newcommand{\Z}{\mathbb Z}
\newcommand{\eps}{\varepsilon}
\newcommand{\A}{\mathcal A}

\newcommand{\E}{\mathbb{E}}
\renewcommand{\P}{\mathbf{P}}
\newcommand{\prob}{\P}
\renewcommand{\Q}{\mathbf{Q}}

\def\Reff{R_{\mathrm{eff}}}
\newcommand{\comb}{\mathrm{Comb}}

\newcommand{\lra}{\leftrightarrow}

\makeatother

\begin{document}
\begin{frontmatter}

\title{Nonconcentration of return times}
\runtitle{Nonconcentration of return times}

\begin{aug}
\author[A]{\fnms{Ori} \snm{Gurel-Gurevich}\ead[label=e1]{origurel@math.ubc.ca}}
\and
\author[A]{\fnms{Asaf} \snm{Nachmias}\corref{}\ead[label=e2]{asafnach@math.ubc.ca}}
\runauthor{O. Gurel-Gurevich and A. Nachmias}
\affiliation{University of British Columbia}
\address[A]{Department of Mathematics \\
University of British Columbia\\
121-1984 Mathematics Rd\\
Vancouver, British Columbia\\
Canada V6T 1Z2\\
\printead{e1}\\
\hphantom{E-mail: }\printead*{e2}} 
\end{aug}

\received{\smonth{9} \syear{2010}}
\revised{\smonth{4} \syear{2012}}

%
\begin{abstract}
We show that the distribution of the first return time $\tau$ to the
origin, $v$, of a simple random walk on an infinite recurrent graph is
heavy tailed and nonconcentrated. More precisely, if $d_v$ is the
degree of $v$, then for any $t\geq1$ we have
\[
\P_v(\tau\ge t) \ge\frac{c }{ d_v \sqrt{t}}
\]
and
\[
\P_v(\tau= t \mid\tau\geq t) \leq\frac{C \log( d_v t) }{ t}
\]
for some universal constants $c>0$ and $C<\infty$. The first bound is
attained for all $t$ when the underlying graph is $\Z$, and as for the
second bound, we construct an example of a recurrent graph $G$ for
which it is attained for infinitely many $t$'s.

Furthermore, we show that in the \textit{comb} product of that graph $G$
with $\Z$, two independent random walks collide infinitely many times
almost surely. This answers negatively a question of Krishnapur and
Peres
[\textit{Electron. Commun. Probab.} \textbf{9} (2004) 72--81]
who asked whether every comb product of two infinite
recurrent graphs has the finite collision property.
\end{abstract}

%
\begin{keyword}[class=AMS]
\kwd{05C81}
\end{keyword}
\begin{keyword}
\kwd{Random walks}
\kwd{return times}
\kwd{finite collision property}
\end{keyword}

\end{frontmatter}

\section{Introduction}
\subsection{Return times}
In this paper we study the distribution of return times of a simple
random walk $X_t$ on an infinite connected graph $G=(V,E)$ with finite
degrees. For $v\in V$, the \textit{hitting time} of $v$ by $X$, denoted
$\tau_v$, is defined by $\tau_v = \min\{ t \geq1\dvtx  X_t = v \}$. When
$X$ starts at $v$ (i.e., $X_0=v$), we call $\tau_v$ the \textit{return
time} to $v$. As usual, the law of $X$ when $X_0=v$ is denoted by
$\P_v$. Our main result is that on any graph these times are heavy
tailed, with exponent at most ${1}/{2}$, and nonconcentrated.
%
\begin{theorem}\label{mainthm1} Let $G=(V,E)$ be an infinite connected
graph with finite degrees $\{d_v\}_{v\in V}$. There exists a universal
constant $c>0$ such that for any $t\geq1$ we have
\[
\P_v(\tau_v \ge t) \ge\frac{c}{d_v \sqrt{t}}.
\]
\end{theorem}
%
\begin{theorem}\label{mainthm2} Let $G=(V,E)$ be an infinite connected
graph with finite degrees $\{d_v\}_{v\in V}$. There exists a universal
constant $C<\infty$ such that for any $t\geq1$ we have
\[
\P_v(\tau_v = t \mid\tau_v \ge t) \le
\frac{C \log(d_v
t)}{t}.
\]
\end{theorem}

The proof of Theorem~\ref{mainthm1} uses electrical network and
martingale arguments, and the proof of Theorem~\ref{mainthm2}
incorporates spectral decomposition of killed random walks. These two
inequalities are sharp up to multiplicative constants. Indeed, for
Theorem~\ref{mainthm1}, it is easy to see that in a copy of $\N$,
together with $d-1$ new vertices who are attached only to $0$ we have
$\prob_0(\tau_0\ge t) \approx\frac{c }{ d \sqrt{t}}$.

Constructing a graph which saturates the inequality of Theorem \ref
{mainthm2} is harder, and we perform this in Section~\ref{secsharp}.
The sharpness of Theorem~\ref{mainthm2} is perhaps more surprising
since most natural examples exhibit an upper bound of order $1/t$. For
example, in $\Z$ it is classical (see~\cite{Feller2}) that $\prob_v(\tau_v \geq t) \approx t^{-1/2}$ and $\prob_v(\tau_v = t)\approx t^{-3/2}$.
It is likely that if the distribution of $\tau_v$ is regular varying in
some sense it is possible to prove a $1/t$ upper bound. Indeed, in the
construction in Section~\ref{secsharp} the rate of decay of $\prob_v(\tau_v \geq t)$ has extremely different behavior at different scales
of~$t$.

It is a well-known fact that $\E\tau_v= \infty$ for any infinite
connected graph. This of course follows from Theorem~\ref{mainthm1},
but a simpler way to see it is to consider the Green function
\[
g(u) = \E_v \sum_{t=1}^{\tau_v} {
\mathbf1}_{\{X_t = u\}},
\]
that is, the expected number of visits to $u$ before returning to $v$.
It is easy to check that the vector $\{g(u)\}_{u\in G}$ is invariant
under the random walk operator and that $g(v)=1$. Hence, $g(u)=d_u/d_v$
for all $u$ in the connected component of $v$. Furthermore, it is clear
that $\sum_{u} g(u) = \E\tau_v$, and since $G$ is connected and
infinite we deduce that $\E\tau_v = \infty$.

\subsection{The finite collision property}\label{finitecolsec}
The construction of Section~\ref{secsharp} is related to the finite
collision property. Recall that an infinite graph $G$ has the
\textit{finite collision property} if two simple random walks $X_t$ and
$Y_t$ collide only finitely many times almost surely, that is, the set
$\{t\dvtx X_t = Y_t\}$ is almost surely finite. It is not hard to see,
using reversibility, that any bounded-degree transient graph has the
finite collision property, and it is an easy exercise to check that
$\Z$ and $\Z^2$ do not have the finite collision property. In fact, any
transitive recurrent graph does not have the finite collision property
(to see this, note that in a transitive graph the number of collisions
has a geometric distribution, hence it is a.s. finite if and only if it
has finite mean, and this mean is finite if and only if the graph is
transient).

It is a surprising discovery of Krishnapur and Peres~\cite{KP} that
there exist recurrent graphs with the finite collision property. In
these graphs, both random walks visit every vertex infinitely often,
but only collide finitely many times. Their constructions involve the
\textit{comb} product of two graphs and is defined as follows. Given two
graphs $G$, $H$ and a vertex $v\in H$, define $\mathrm{Comb}_v(G,H)$ to be the
graph with vertex set $V(G) \times V(H)$ and edge set
\[
\bigl\{ \bigl\{(x,w), (x,z)\bigr\}\dvtx  \{w,z\} \in E(H), x\in V(G) \bigr\} \cup
\bigl\{ \bigl\{ (x,v), (y,v) \bigr\}\dvtx  \{x,y\} \in E(G) \bigr\}.
\]

Krishnapur and Peres prove in~\cite{KP} that $\mathrm{Comb}_0(G,\Z)$ and
$\mathrm{Comb}_{(0,0)}(G, \Z^2)$ have the finite collision property whenever $G$
is an infinite recurrent graph with bounded degrees. They asked (see
first question of Section 4 of~\cite{KP}) whether $\mathrm{Comb}_v(G,H)$ has
the finite collision property whenever $G$ and $H$ are infinite
recurrent graphs. Our next result answers their question negatively.
%
\begin{theorem} \label{nocollisionthm} There exists a bounded-degree,
connected, infinite graph $H$ and a vertex $v\in H$ such that
$\mathrm{Comb}_{v}(\Z, H)$ does not have the finite collision property.
\end{theorem}

We do not use Theorems~\ref{mainthm1} and~\ref{mainthm2} for the proof
of Theorem~\ref{nocollisionthm}; however, the graph for which Theorem \ref
{mainthm2} is saturated (see Section~\ref{secsharp}) is the graph $H$
in the statement of the above theorem. The important property of this
graph is that, roughly, at certain scales it behaves like a finite
graph. This property is crucial both for showing the sharpness of
Theorem~\ref{mainthm2} and for the proof of Theorem~\ref{nocollisionthm}.

In fact, general results in this flavor have recently been obtained.
Barlow, Peres and Sousi~\cite{BPS} give a general condition for a graph
not to have the finite collision property. While this condition fails
for the graph constructed in the proof of Theorem~\ref{nocollisionthm},
they use it to show that various natural graphs with fractal geometry
do not have the finite collision property.

\subsection{Extensions and questions} Theorems~\ref{mainthm1} and
\ref
{mainthm2} can be extended to the setting of finite graphs. Indeed, the
proofs of both theorems can be extended so that they hold for any
finite graph and any $t \leq R^2$, where $R$ is the effective
resistance diameter of the graph $R=\max_{v,u} \Reff(v\lra u)$. These
extensions to the proof are straightforward. In particular cases it is
even possible to prove stronger assertions; see, for example, Lemma
\ref{expander}.

We cannot expect Theorems~\ref{mainthm1} and~\ref{mainthm2} to hold for
general hitting times. Indeed, if $u$ is a vertex such that its removal
leaves $v$ in a finite component (these are sometimes called \textit{cutpoints}), then the distribution of $\tau_u$ started from $v$ has
exponential decay since as much as the distribution of $\tau_u$ started
from $v$ is concerned, the graph is finite. However, perhaps there is
hope to prove similar estimates when $u$ is not such a cutpoint.

To demonstrate that Theorem~\ref{mainthm2} does not hold for hitting
times in general, consider the following example: the graph is simply
the natural numbers, with $2^{2^n}$ edges between $n$ and $n+1$. If a
simple random walk starts at $0$, there is a positive probability it
will never take a step backward, that is, $X_i=i$ for all $i$. This
means that $\P_0(\tau_n=n \mid \tau_n \ge n)$ does not decay to $0$. Of
course, this graph has unbounded degrees, so it remains to see whether
a bounded degree example exists. Similar questions can be also asked
about \textit{commute times}, that is, the first time to hit some
specific vertex and return to the origin. These retain some of the
symmetry of return times, and perhaps Theorems~\ref{mainthm1} and \ref
{mainthm2} can be extended to them.

Finally, is it true that for any $t \geq1$ the graph $\N$ minimizes
the quantity $\prob_0(\tau_0 \geq t)$ of all connected infinite graphs
with the origin having degree 1?

\subsection{Notation} We say $f(r) \approx g(r)$ when there exists a
constant $C$ such that $C^{-1} f(r) \leq g(r) \leq C f(r)$. We denote
by $C$ and $c$ positive constants where $c$ will usually denote a
``small enough'' constant and $C$ a ``large enough'' constant. The
values of $C$ and $c$ will change occasionally, even within the same
formula. We will not be strict about assigning noninteger values to
integer variable, and when doing so we always assign the floored value.

\section{\texorpdfstring{Proof of Theorems \protect\ref{mainthm1} and \protect\ref{mainthm2}}
{Proof of Theorems 1.1 and 1.2}}\label{secmainthm}
We begin with a few lemmas. For background about effective resistance
we refer the reader to~\cite{LP}.
%
\begin{lemma}\label{getfar} Let $G$ be a finite graph. For any two
vertices $x,y$ and any $\eps>0$ we have
\[
\P_x\bigl(\tau_y \leq\eps\bigl(\Reff(x \lra y)
\bigr)^2\bigr) \leq\eps,
\]
where $\Reff(x \lra y)$ is the effective resistance between $x$ and
$y$, when $G$ is considered as an electric network with unit resistances.
\end{lemma}
\begin{pf}
Let $f\dvtx  G \to\R_+$ be the potential corresponding to a unit current
flow of the electrical network between $x$ and $y$. That is, $f$ is the
harmonic function on $G \setminus\{x,y\}$ with boundary values
$f(x)=0$ and $f(y)=\Reff(x \lra y)$ (as $G$ is finite $f$ is uniquely
determined). The associated unit current flow is an antisymmetric
function on directed edges $i\dvtx E(G)\to\R$ such that:
\begin{longlist}[(iii)]
\item[(i)] $\sum_{u \sim y} i(yu) =1$;
\item[(ii)] for any $u \in G \setminus\{x,y\}$, we have $\sum_{v
\sim
u} i(uv) = 0$;
\item[(iii)] for any oriented cycle $e_1,\ldots, e_m$, we have $\sum_{1 \leq j \leq m} i(e_j) = 0$.
\end{longlist}
Since we have unit edge resistances we get that $i(uv) = f(u)-f(v)$ for
any edge $uv$. We first observe that $f$ is a contraction; that is, for
any edge $uv$ we have $|f(u)-f(v)|\leq1$. Indeed, assume without loss
of generality that $f(v) < f(u)$ and let $s>0$ be a number such that
$f(v) < s \leq f(u)$. Consider the cut $(S,S^c)$ defined by $S=\{ w\dvtx
f(w) > s\}$. The sum of the unit current flow $i$ on edges leading from
$S$ to $S^c$ is $1$ and each edge receives nonnegative flow, hence
$f(u)-f(v)=i(uv)\le1$ (another way to see this is combining
Proposition 2.2 and Exercise 2.31 of~\cite{LP}). We deduce that
\[
\E \bigl[ f^2(X_t) - f^2(X_{t-1}) \mid
X_{t-1} \bigr] = \E \bigl[ \bigl(f(X_t) -
f(X_{t-1})\bigr)^2 \mid X_{t-1} \bigr] \leq1,
\]
when $X_{t-1} \neq y$, and hence$f^2(X_{t \wedge\tau_y})-t\wedge\tau_y$ is a supermartingale. Put $T=\eps\Reff(x \lra y)^2$. Optional
stopping yields that
\[
\E_x \bigl[ f^2(X_{T \wedge\tau_y}) \bigr] \leq
\E_x[T \wedge\tau_y] \leq T.
\]
If $\tau_y < T$, then $f^2(X_{T \wedge\tau_y}) = \Reff(x \lra y)^2$.
Thus, by Markov's inequality we get
\[
\P_x (\tau_y < T) \leq\frac{T }{\Reff(x \lra y)^2} \leq\eps,
\]
concluding the proof of the lemma.
\end{pf}


%
\begin{pf*}{Proof of Theorem~\ref{mainthm1}}
We prove the assertion
with $c=\frac14$. For $r>0$ we write $B(v,r)$ for the ball of radius
$r$ in $G$ according to the shortest path metric and write $\partial
B(v,r)$ for its boundary, that is, $\partial B(v,r) = B(v,r)\setminus
B(v,r-1)$. We consider the effective resistance $\Reff(v \lra\partial
B(v,r))$. Fix $t \geq1$. If for all $r>0$ we have that $\Reff(v \lra
\partial B(v,r)) \leq4 \sqrt{t}$ (this can only happen in the
transient case), then
\[
\P_v(\tau_v \geq t) \geq\lim_{r \to\infty}
\P_v \bigl(X_t \mbox{ hits } \partial B(v,r) \mbox{
before } v \bigr) \geq\frac{1 }{4 d_v \sqrt {t}}.
\]
Otherwise, let $r$ be the first radius such that $\Reff(v \lra
\partial
B(v,r)) \geq4\sqrt{t}$. As in the proof of Lemma~\ref{getfar} let $f$
be the harmonic function on $B(v,r)$ with $f(v)=0$ and $f(\partial
B(v,r))=\Reff(v \lra\partial B(v,r))$. Let $S$ be the set of vertices
$S= \{ u\dvtx  f(u) \leq2\sqrt{t} \}$. We saw in the proof of Lemma \ref
{getfar} that $f$ is a contraction. Hence, any vertex $x \in N(S)$,
where $N(S)$ denotes the neighbors of $S$ which are not in $S$, has
$2\sqrt{t} \leq f(x)\leq2\sqrt{t} + 1$. Let $f^+$ and $f^-$ be the
harmonic functions on $S \cup N(S)$ such that $f^+(v)=f^-(v)=0$ and
$f^+(N(S))=2\sqrt{t}+1$ and $f^-(N(S))=2\sqrt{t}$. The maximum
principle gives that $f^-(x) \leq f(x) \leq f^+(x)$ for all $x\in S
\cup N(S)$, and hence the total current flow associated with $f^+$
($f^-$) is larger (smaller) than~$1$. Therefore,
\[
2\sqrt{t} \leq\Reff\bigl(v \lra N(S)\bigr) \leq2\sqrt{t}+1.
\]
%
We have that
%
\begin{equation}
\label{hitfar} \P_v ( \tau_{N(S)}< \tau_v ) =
\frac{1 }{ d_v
\Reff(v \lra N(S))} \geq\frac{1 }{ d_v (2\sqrt{t}+1)},
\end{equation}
where $\tau_{N(S)}=\min\{t\ge1 \mid X_t \in N(S)\}$. The strong Markov
property implies
\[
\P_v(\tau_v \geq t) \geq\P_v (
\tau_{N(S)}<\tau_v ) \min_{u\in
N(S)}\P_u (
\tau_v \ge t ).
\]
To estimate the second probability on the right-hand side we apply
Lemma~\ref{getfar} with $\eps=1/4$. We deduce that this probability is
at least $3/4$. This together with (\ref{hitfar}) gives that
\[
\P_v(\tau_v \geq t) \geq\frac{3 }{4 d_v (2 \sqrt{t}+1)} \ge
\frac{1 }{
4d_v \sqrt{t}},
\]
concluding our proof.
\end{pf*}

%




The following is a well-known lemma in the context of return
probabilities. We include its proof here for completeness, and since we
were unable to find it in the literature in the context of \textit{first}
return probabilities.
%
\begin{lemma}[(Spectral decomposition)]\label{spectra} Let
$G=(V,E)$ be an infinite connected graph with finite degrees, and let
$v\in V$. Then there exists a finite measure $\mu$ on $[-1,1]$ such
that for all $t\geq2$ we have
\[
\P_v(\tau_v = t) = \int_{-1}^1
x^{t-2} \,d \mu.
\]
\end{lemma}
\begin{pf} By conditioning on the location of the random walk at time
$\lceil t/2 \rceil$ and using the Markov property, we get that
\[
\P_v(\tau_v = t) = \sum_{u \ne v}
\P_v \bigl(X_{\lceil t/2 \rceil} = u, \tau_v \geq\lceil t/2
\rceil \bigr)\P_u \bigl(\tau_v=\lfloor t/2 \rfloor
\bigr).
\]
Observe that by the reversibility property of the simple random walk,
we have
\[
\P_v \bigl(X_{\lceil t/2 \rceil}=u, \tau_v \geq\lceil t/2
\rceil \bigr)=\frac{d_u }{ d_v} \P_u \bigl(\tau_v=
\lceil t/2 \rceil \bigr)
\]
and hence
%
\begin{equation}
\label{expand} \P_v(\tau_v = t) = \frac1{d_v}
\sum_{u \ne v} d_u \P_u \bigl(
\tau_v=\lceil t/2 \rceil \bigr) \P_u \bigl(
\tau_v=\lfloor t/2 \rfloor \bigr).
\end{equation}

Consider the Hilbert space $\ell_2(G)$ of functions from $V\setminus\{
v\}$ to $\R$ equipped with the inner product
\[
\langle f, g \rangle= \sum_{u \ne v} d_u
f(u) g(u)
\]
and the corresponding norm.
Let $\Q$ be the random walk operator killed upon hitting~$v$. That is,
\[
\Q f (u)=\frac1{d_u} \sum_{w \sim u, w\ne v} f(w).
\]

One can easily check that $\Q^t f(u) = \E_u (f(X_t) {\mathbf1}_{\{\tau
_v >
t\}})$. Define the function $h(w)=\P_w(\tau_v=1)$ [i.e., $h(w)=1/d_w$
if $w \sim v$ and $h(w)=0$ otherwise]. We have that
\[
\Q^{t-1} h (u) = \P_u(\tau_v=t).
\]
Hence, we can rewrite equation (\ref{expand}) as
\[
\P_v(\tau_v = t)= \frac1{d_v} \bigl
\langle\Q^{\lceil t/2 \rceil-1} h, \Q^{\lfloor t/2 \rfloor-1} h \bigr\rangle.
\]

A simple calculation shows that
\[
\langle\Q f, g \rangle= \sum_{u \ne v} \sum
_{w \sim u, w\ne v} f(w) g(u) = \langle f, \Q g \rangle,
\]
that is, $\Q$ is self-adjoint with respect to the inner product of
$\ell_2(G)$. Hence, we may apply the spectral theorem (see~\cite{RUDIN} or
\cite{HA}) and conclude that there exists a measure $\nu$ on a space
$\Omega$ and some $\lambda\in L_2(\nu)$ such that $\Q$ is
isometrically equivalent to multiplication by $\lambda$,
%
\[
\P_v(\tau_v = t) = \frac1{d_v} \bigl
\langle\Q^{\lceil t/2 \rceil-1} h, \Q^{\lfloor t/2 \rfloor-1} h \bigr\rangle=
\frac1{d_v} \int_{\Omega} \lambda^{t-2}(
\omega) \hat{h}^2(\omega) \,d \nu(\omega),
\]
where $\hat{h}$ is the image of $h$ under the isometry. Since $Q$ is
self-adjoint and substochastic, $\lambda$ takes only real values in
$[-1,1]$ (up to $\nu$ null sets). If we define $\mu$ to be the
pull-back measure
\[
\mu(A) = \frac1{d_v} \int_{\lambda^{-1}(A)}
\hat{h}^2(\omega) \,d\nu (\omega)
\]
for any Borel set $A \subset[-1,1]$, then we get that
\[
\P_v(\tau_v = t) = \int_{-1}^{1}
x^{t-2} \,d\mu,
\]
which completes the proof.
\end{pf}

%
\begin{pf*}{Proof of Theorem~\ref{mainthm2}}
We prove the assertion with $C=50$. We assume that $t-1 \geq48 \log(48
d_v)$. Otherwise we have that $t \leq50 \log(50 d_v)$ and then either
$50 t^{-1}\log(d_v t) \geq1$ or $d_v=1$ and $t=1$, and the assertion is
trivial in both cases. Lemma~\ref{spectra} gives that $\P_v(\tau_v = t)
= \int_{[-1,1]} x^{t-2} \,d\mu$ for some finite measure $\mu$. Write
$A\subset[-1,1]$
for the set 
\[
A = \biggl\{ x\dvtx  |x| \geq1-\frac{4\log d_v t }{ t} \biggr\}.
\]
Assume first that $t$ is even. In this case we may bound
%
\begin{equation}
\label{technicallower}\qquad \P_v(\tau_v \ge t) \geq\sum
_{j \geq0} \P_v(\tau= t+2j) = \int
_{[-1,1]} \frac{x^{t-2}}{1-x^2} \,d\mu\geq\int_A
\frac{x^{t-2}}{1-x^2} \,d\mu.
\end{equation}
Thus
%
\begin{equation}
\label{lastbound} \frac{\P_v(\tau_v = t) }{\P_v(\tau_v \ge t)} \leq
\frac{\int_{A^c} x^{t-2} \,d\mu}{\P_v(\tau_v \ge t)} +
\frac{\int_A x^{t-2}
\,d\mu}{\int_A ({x^{t-2}}/({1-x^2})) \,d\mu}.
\end{equation}
If $x \notin A$, then $x^{t-2} \le(1-\frac{4\log(d_v t) }{ t})^{t-2}
\le e^{-2 \log(d_v t)} = (d_v t)^{-2}$ since $t\geq4$. We also have
$\mu([-1,1]) \leq1$ by putting $t=2$ in Lemma~\ref{spectra}. Hence, by
Theorem~\ref{mainthm1} (recall that we proved it with $c=\frac14$), we
get that
\[
\frac{\int_{A^c} x^{t-2} \,d\mu}{\P_v(\tau_v \ge t)} \leq\frac{4 }{ t}.
\]
If $x \in A$, then $x^2 \geq1- 8t^{-1} \log d_v t$, and hence
\[
\frac{\int_A x^{t-2} \,d\mu}{\int_A ({x^{t-2}}/({1-x^2})) \,d\mu} \leq\frac{8
\log d_v t }{ t}.
\]
We put these two in (\ref{lastbound}) and get that $\P_v(\tau_v=t
\mid
\tau_v\geq t)\leq\frac{12 \log d_v t }{ t}$ when $t$ is even. When
$t>1$ is odd (when $t=1$ the assertion is trivial), we first bound
\[
\P_v(\tau_v = t) = \int_{[-1,1]}
x^{t-2} \,d\mu\leq\int_{[-1,1]} x^{t-3} \,d\mu=
\P_v(\tau_v=t-1).
\]
By the assertion for even $t$'s we get that
\[
\P_v(\tau_v = t-1) \leq\frac{12 \log(d_v(t-1)) }{ t-1}
\P_v(\tau_v \geq t-1).
\]
Also, $\P_v(\tau_v \geq t) = \P_v(\tau_v \geq t-1) - \P_v(\tau_v =
t-1)$ so
\[
\P_v(\tau_v \geq t) \geq \biggl( 1 -
\frac{12 \log(d_v(t-1)) }{ t-1} \biggr) \P_v(\tau_v \geq t-1).
\]
Hence
\[
\frac{\P_v(\tau_v = t) }{\P_v(\tau_v \geq t)} \leq \biggl( 1 - \frac{12
\log
(d_v(t-1)) }{ t-1} \biggr)^{-1}
\frac{\P_v(\tau_v=t-1) }{\P_v(\tau_v
\geq t-1)},
\]
whenever $12 \log(d_v(t-1))/(t-1)<1$. Furthermore, whenever $t-1 \geq
48\* \log(48 d_v)$ we have that $12 \log(d_v(t-1))/(t-1)<1/2$ [since for
any $\eps\in(0,e^{-1})$ and $x \geq\eps^{-1} \log\eps^{-1}$ we have
$x^{-1}\log x \leq2\eps$], so
\[
\frac{\P_v(\tau_v = t) }{\P_v(\tau_v \geq t)}\leq\frac{24 \log d_v (t-1)
}{ t-1 } \leq\frac{50 \log d_v t }{ t },
\]
completing our proof.
\end{pf*}

\section{Preliminaries on expanders}

Recall that a family $\{G_n\}$ of $d$-regular graphs on $n$ vertices is
called an \textit{expander} family if there is some constant $\rho<1$ such
that the second largest eigenvalue in absolute value of the transition
matrix $\lambda_2(n)$ satisfies $|\lambda_2(n)| \leq\rho$ for all $n$.
The quantity $1-\rho>0$ is called the \textit{absolute spectral gap} of
the sequence $\{G_n\}$. Note that in particular this implies that $G_n$
is not bipartite, and the simple random walk on it is not periodic. It
is a classical fact (see Theorem 6.9 in~\cite{LP}) that if $\{X_t\}$ is
a simple random walk on $G_n$, then for any $v\in G_n$ and any integer
$t$ we have
%
\begin{equation}
\label{mixing} \biggl|\P(X_t = v) - \frac{1 }{ n} \biggr| \le
e^{-(1-\rho)t}.
\end{equation}
Another useful fact (see~\cite{BK}) is that if we put unit resistance
on each edge of the expander, then there exists a constant $C=C(\rho
)<\infty$ such that for any $u,v \in G_n$ the effective resistance satisfies
%
\begin{equation}
\label{reff} \Reff(u \lra v) \leq C.
\end{equation}

In the following four lemmas we study the simple random walk on the
graph $G$ obtained by taking a $d$-regular expander and an arbitrary
vertex $v$ and adding a new vertex $v'$ together with the edge $\{v',v\}
$. We consider $d$ as fixed and $|G|=n$ tending to infinity (in all our
applications taking $d=3$ suffices). All the constants in the following
lemmas depend on $\rho$ but not on~$n$.
%
\begin{lemma}\label{expander} There exists a constant $\delta=\delta
(\rho)>0$ such that for any \mbox{$u\neq v'$}
\[
\P_u ( \tau_{v'} \geq\delta n ) \geq\delta
\]
and
\[
\P_u ( \tau_{v'} \leq n ) \geq\delta.
\]
\end{lemma}
\begin{pf} We begin by proving a lower bound on $\prob(\tau_{v'} \geq
\delta n)$. Since the walker must visit $v$ in order to visit $v'$, it
suffices to prove the assertion for $u=v$. Since $G$ has bounded
degree,\vadjust{\goodbreak} there exists a vertex $y\in G$ with graph distance from $v$ at
least $c \log n$. By (\ref{reff}) the effective resistance between $v$
and $y$ is bounded by a constant and hence with at least constant
positive probability $X_t$ hits $y$ before $v$. We deduce that for some
constant $c>0$ we have
%
\begin{equation}
\label{expanderfirst} \P( \tau_{v'} \geq c \log n ) \geq c.
\end{equation}
Furthermore, by (\ref{mixing}) and the union bound we have that
\[
\P \bigl( \exists t \in[c\log n, \delta n] \mbox{ with } X_t = v
\bigr) \leq\delta+ \frac{e^{-(1-\rho) c \log n} }{1-e^{-(1-\rho)}},
\]
where $\rho<1$ is the uniform bound on the second eigenvalue. This
together with (\ref{expanderfirst}) shows that $\prob(\tau_{v'} \geq
\delta n) \geq\delta$ for some constant $\delta>0$.

To prove a lower bound on $\prob(\tau_{v'} \leq n )$ we employ a
second moment calculation. Write $Y$ for the number of visits to $v'$
before time $n$. It is clear by (\ref{mixing}) that $\prob(X_t = v')
\geq\frac{1 }{2n}$ for any $t \geq C\log n$ so $\E Y \geq c$ for some
$c>0$. On the other hand, if $t_2 > t_1$ and $X_{t_1}=v'$, then by
(\ref
{mixing}) the probability of having $X_{t_2}=v'$ is at most $n^{-1} +
e^{-c(t_2-t_1)}$ for some $c>0$. This gives that $\E Y^2 \leq C$, and
we get that $Y>0$ with some fixed probability by the inequality
\[
\prob(\mathbf{X} > 0) \geq\frac{ (\E\mathbf{X})^2 }{\E\mathbf
{X}^2 },
\]
valid for any nonnegative random variable $\mathbf{X}$. This completes
the proof.
\end{pf}
%
\begin{lemma} \label{mixingstopping} There exist constants $C=C(\rho
)>0$ and $c=c(\rho)>0$ such that for any vertex $u \neq v'$ there
exists a set of vertices $S_u$ such that $|S_u|=n-o(n)$ and for any
$w\in S_u$ and any $C\log n \le t \le n$
\[
\P_u(X_t=w, \tau_{v'} \geq t) \ge
\frac{c}{n}.
\]
\end{lemma}
\begin{pf} For any two vertices $u,w$ and any $C \log n \leq t \leq n$,
we have that $\P_u(X_t=w)\le2/n$ by (\ref{mixing}) and $\P_u(\tau_{v'}\ge t) \ge\P_u(\tau_{v'} \ge n) \ge\delta^{\lceil\delta^{-1}
\rceil} >0$ by iterating Lemma~\ref{expander}. Hence
%
\begin{equation}
\label{outofexpander1}\P_u(X_t=w \mid
\tau_{v'}\ge t) \le\frac
{C}{n}
\end{equation}
for some $C=C(\rho)>0$. Furthermore, $\P_\pi(\tau_{v'}\le C \log n) =
O(n^{-1} \log n)$, where $\pi$ is the stationary distribution. This is
because the expected number of visits to $v'$ by time $C\log n$ is
$O(n^{-1}\log n)$. Define
\[
S= \bigl\{ u\dvtx  \P_u(\tau_{v'}\le C \log n) \le C
n^{-1} \log^2 n \bigr\},
\]
and we deduce that $|S| \ge n(1- \log^{-1} n)$. We combine this with
(\ref{outofexpander1}) to get that
%
\begin{equation}
\label{outofexpander2}\P_u(X_t \in S \mid
\tau_{v'}\ge t) \ge1- \frac{C
}{\log n}.
\end{equation}
By the definition of $S$ and (\ref{mixing}), for any $u \in S$ and any
$w$, we have
\[
\P_u(X_{C \log n}=w \mid\tau_{v'}\ge C \log n) \le
\frac{1+o(1)}{n}.
\]
Thus, by the Markov property, for any $u$ and $w$,
\[
\P_u(X_{t+C \log n}=w \mid X_t \in S,
\tau_{v'}\geq t+ C\log n) \le\frac{1+o(1)}{n}
\]
and therefore there exists a set $S_u$ such that $|S_u|=n-o(n)$ such
that for every $w \in S_u$, we have
\[
\P_u(X_{t+C \log n}=w \mid X_t \in S,
\tau_{v'}\geq t+ C\log n) \ge\frac1n\bigl(1-o(1)\bigr).
\]
This together with (\ref{outofexpander2}) shows that for any $w\in S_u$
and $2 C \log n \leq t \leq n$, we have
%
\begin{equation}
\label{outofexpander3}\P_u(X_t=w, \tau_{v'}>
t) \ge\frac
{c}{n},
\end{equation}
completing our proof.
\end{pf}
%
\begin{lemma}\label{outofexpander} There exist constants $C=C(\rho)>0$
and $c=c(\rho)>0$, such that for every $C \log n \leq t \leq n$ and any
$u\neq v'$, we have
\[
\P_u( \tau_{v'} = t ) \ge\frac{c}{n}.
\]
\end{lemma}
\begin{pf}
Reversibility of the simple random walk implies that
\begin{eqnarray*}
\P_u(\tau_{v'}=t)
&\geq&
\frac{1 }{ d+1} \sum
_{w\neq v'} \P_u\bigl(X_{\lceil t/2
\rceil}=w,
\tau_{v'} \geq\lceil t/2 \rceil\bigr)\\
&&\hspace*{47.5pt}{}\times \P_{v'}
\bigl(X_{\lfloor t/2
\rfloor}=w, \tau_{v'} \geq\lfloor t/2 \rfloor\bigr),
\end{eqnarray*}
since the maximum degree in $G$ is $d+1$. The assertion now follows
from plugging in Lemma~\ref{mixingstopping} and summing.
\end{pf}

Our last lemma about expanders concerns two independent simple random
walks $X_t$ and $Y_t$. We denote by $\P_{u_1,u_2}$ for the probability
distribution generated when $X_0=u_1$ and $Y_0=u_2$. We denote $\tau_u^X$ for the hitting time of $X_t$ of $u$ and similarly for $Y$.
%
\begin{lemma} \label{collisioninsideexpander} There exists a constant
$c=c(\rho)>0$ such that for any $u_1\neq v'$ and $u_2 \neq v'$,
\[
\P_{u_1,u_2} \bigl( \exists t \le n \wedge\tau^X_{v'}
\wedge\tau^Y_{v'} \mbox{ such that } X_t =
Y_t \bigr) \geq c.
\]
In other words, the probability that $X_t$ and $Y_t$ collide before
time $n$ and before either of them hits $v'$ is uniformly positive.
\end{lemma}
\begin{pf} For any $C\log n \le t \le n$, by Lemma~\ref{mixingstopping}
there exists a constant $c>0$ and a set $S$ of size $|S|=n-o(n)$ such
that for any $w \in S$,
\[
\P_{u_1}\bigl(X_t = w, \tau_{v'}^X
\geq t\bigr) \geq\frac{c }{ n},\qquad \P_{u_2}\bigl(Y_t
= w, \tau_{v'}^Y \geq t\bigr) \geq\frac{c }{ n}.
\]
Hence
\[
\sum_{w \in G} \P_{u_1} \bigl(X_t
= w, \tau_{v'}^X \geq t \bigr) \P_{u_2}
\bigl(Y_t = w, \tau_{v'}^Y \geq t \bigr) \geq
\frac{c
}{ n}.
\]
Let $N = |\{ t \le n \wedge\tau^X_{v'} \wedge\tau^Y_{v'}\dvtx  X_t =
Y_t\}
|$; then by the previous inequality and the independence of $X_t$ and
$Y_t$ we learn that $\E N \geq c$. To bound the second moment of $N$ by
(\ref{mixing}) we have $\prob_v(X_t = u) \leq n^{-1} + e^{-ct}$, for
some $c>0$. We deduce by the Markov property that for any $t_2 > t_1$,
we have that
\[
\P_{u_1,u_2} ( X_{t_1} = Y_{t_1} \mbox{ and }
X_{t_2}=Y_{t_2} ) \leq \bigl(n^{-1} +
e^{-c(t_2-t_1)}\bigr) \prob_{u_1,u_2} ( X_{t_1} =
Y_{t_1}).
\]
Similar considerations give that $\P_{u_1,u_2} ( X_{t_1} = Y_{t_1})
\leq n^{-1} + e^{-ct_1}$, and so we have that
\[
\E N^2 \le\sum_{t_1=1}^n \sum
_{t_2=t_1}^n \P_{u_1,u_2} (
X_{t_1} = Y_{t_1} \mbox{ and } X_{t_2}=Y_{t_2}
) \leq C
\]
for some constant $C>0$ and the assertion of the lemma follows.
\end{pf}

\section{Sharpness}\label{secsharp}

In this section we show that the estimate of Theorem~\ref{mainthm2} is
sharp up to the multiplicative constant $C$. In order to elucidate the
ideas of the construction we begin with a simple construction showing
the sharpness of Theorem~\ref{mainthm2} for a single $t$. We then
construct a more elaborate graph for which the theorem is sharp for an
infinite sequence of $t$'s. This graph will be useful later in Section
\ref{combssec}---it will be the base of the comb for the construction
of Theorem~\ref{nocollisionthm}.

\subsection{A simple construction} Given an integer $t$ we construct
the graph $G_t$ as follows. Let $\{E_i\}_{i \geq1}$ be a sequence of
disjoint $3$-regular expanders with spectral gap $1-\rho>0$ and
$|E_i|=4i$. Let $\delta= \delta(\rho)>0$ be the constant from Lemma
\ref{expander}, and take
\[
n = \frac{3 \log(1/\delta) t}{\delta\log t}.
\]
The graph $G_t$ is constructed by taking $\N=\{0,1,\ldots\}$ with edges
between consecutive numbers, and attaching to $0$, by an edge, the
graph $E_n$ (the degree of $0$ is thus~$2$).
%
\begin{theorem}\label{sharp} There exists a constant $c=c(\rho)>0$ such
that the simple random walk on $G_t$ satisfies
%
\begin{equation}
\label{sharpineq} \P_0 ( \tau_0 = t \mid
\tau_0 \geq t ) \geq \frac{c \log t }{ t}.\vadjust{\goodbreak}
\end{equation}
\end{theorem}
\begin{pf} We abbreviate $\tau$ for $\tau_0$ and write $\{X_t\}$ for
the simple random walk on $G_t$ starting at $0$. Write $\A$ for the
event that $X_1=1$, so $\P(\A) = 1/2$. It is a well-known fact (see
\cite{Feller2}) that the probability that a random walk on $\N$ does
not return to the origin in $t$ steps decays like $t^{-1/2}$, that is,
$\P(\tau\geq t \mid\A) \approx t^{-1/2}$. By\vspace*{1pt} iterating
Lemma~\ref{expander} using the Markov property we get that $\P( \tau\geq t \mid
\A^c ) \geq\delta^{{t / \delta n}} \geq t^{-1/2}$ by the definition of
$n$, so
\[
\P(\tau\geq t) \approx\P\bigl( \tau\geq t \mid\A^c \bigr).
\]
Now, we condition on the first $t-\delta n$ steps walk and apply Lemma
\ref{outofexpander} together with the Markov property to get that
\[
\prob\bigl(\tau= t \mid\A^c\bigr) \geq\frac{c }{ n} \prob\bigl(
\tau\geq t - \delta n \mid\A^c\bigr),
\]
hence
\[
\prob\bigl(\tau= t \mid\A^c\bigr) \geq\frac{c }{ n} \prob\bigl(
\tau\geq t \mid\A^c\bigr) \geq\frac{c \log t }{ t} \prob(\tau\geq t),
\]
completing our proof.
\end{pf}

\subsection{The full construction}\label{basegraph} We now construct a
graph saturating inequality (\ref{sharpineq}) for infinitely many
$t$'s. This graph will also be used in the next section as the base
graph (a \textit{tooth}) of the comb exhibiting almost sure infinitely
many collisions. Let $\{E_i\}_{i \geq1}$ be a sequence of disjoint
$3$-regular expanders with spectral gap $1-\rho>0$ and of sizes
$|E_i|=n_i$. Let $\{h_i, n_i\}_{i\geq0}$ be two increasing sequences of
positive integers with $h_1 \geq2C$ where $C=C(\rho)$ is the constant
from (\ref{reff}) and such that
%
\begin{equation}
\label{seqcond} h_i \gg n_{i-1} h_{i-1}^2
\quad\mbox{and}\quad n_i = h_i^{15},
\end{equation}
where $a_i \gg b_i$ means $b_i/a_i \to0$ as $i \to\infty$. For each
$i$ let $v_i \in E_i$ be an arbitrary vertex.
The graph $G=G(\{h_i,n_i\})$ consists of $\N=\{0,1,\ldots\}$ with edges
between consecutive numbers, and we attach the expander $E_i$ by adding
an edge between $v_i$ and $h_i \in\N$.

All the constants in this section will depend on $\rho$ but not on $i$
or $\{h_i,n_i\}$. The following is the main result of this section.
%
\begin{theorem} \label{sharpness} Consider the graph $G(\{h_i,n_i\})$
for $\{h_i,n_i\}$ satisfying (\ref{seqcond}). There exists a constant
$c=c(\rho)>0$ such that for $t_i = ch_in_i \log n_i$ we have
\[
\P_0 ( \tau_0 = t_i \mid
\tau_0 \geq t_i ) \geq\frac{c \log t_i
}{
t_i}.
\]
\end{theorem}

The rough idea of the proof goes as follows (for brevity we omit the
$i$ subscript). The event $\tau_0 \geq t$ occurs mainly when the walk
hits $h$ before $0$ (which happens with probability $h^{-1}$) and then
stays around $h$ and $E$ for about $t$ steps without returning\vadjust{\goodbreak} to 0
[which happens with probability $(1-h^{-1})^{t/n}\approx n^{-c}$, since
there are roughly $t/n$ visits to $h$ and each time the probability of
returning to $0$ before $h$ is $h^{-1}$]. So this event happens with
probability about $h^{-1} n^{-c}$.

Now we need to give a lower bound for $\P(\tau_0 = t)$. Again, the
probability that the walk hits $h$ before $0$ is $h^{-1}$. Assume that
happened, and consider the excursion from 0 to 0. We can partition it
into 3 parts: until we hit $h$, between the first and last visits to
$h$ and after the last visit to $h$ and until we hit 0 again. Call the
lengths of these 3 parts $s_1$, $s_2$ and $s_3$, respectively, and
notice that they are independent. With high probability $s_1$ and $s_3$
are roughly $h^2 \ll t$. Conditioning on the values of $s_1$ and $s_3$,
we want a lower bound for the probability that $s_2=t-s_1-s_3$. The
probability that the walk will stay around $h$ (and $E$) for about $t$
steps is again $n^{-c}$ and the probability it will be at $h$ precisely
at time $t-s_1-s_3$ is of order $n^{-1}$ (since it is roughly mixed in
$E$). Finally, the probability that this is the last visit to $h$ is
$(3h)^{-1}$. Put together, we get a lower bound for $\P(\tau_0 = t)$ of
order $h^{-2} n^{-1-c}$, so the ratio is $h^{-1}n^{-1}\approx
t^{-1}\log t$, as required.


We begin with some preparatory lemmas and observations leading to the
proof of this theorem. In all of the statements below we are
considering a simple random walk on $G(\{h_i,n_i\})$ for $\{h_i,n_i\}$
satisfying (\ref{seqcond}). For a vertex $v$ of $G(\{h_i,n_i\})$ we
write $h(v)$ for its \textit{height}, that is, if $v\in E_i$, then $h(v) =
h_i$ and if $v \in\N$, then $h(v) =v$.
%
\begin{lemma} \label{easyfacts} We have:
\begin{longlist}[(2)]
\item[(1)] For any $h>0$, we have $\prob_0(\tau_h < \tau_0) = h^{-1}$.
\item[(2)] For any $i$ and $v$ such that $0 \leq h(v) < h_i$, we have $\E_v
( \tau_{h_i} \wedge\tau_0 ) \leq2h_i^2$.
\end{longlist}
\end{lemma}
\begin{pf} Part (1) is immediate since the effective resistance between
vertices $0$ and $h$ is precisely $h$. Part (2) follows immediately by
the commute time identity; see~\cite{LP}. Indeed, the effective
resistance between $v$ and $\{0,h_i\}$ is at most $h_i/2+C$ [where $C$
is the constant from (\ref{reff})], and the number of edges in the
subgraph \textit{between} $0$ and $h_i$ [i.e., the subgraph spanned on all
vertices $v$ having $0 \leq h(v) < h_i$ together with $h_i$] is at most
\[
h_i + \sum_{j=1}^{i-1}
n_j \leq2h_i
\]
by condition (\ref{seqcond}). We conclude the proof since $h_i \geq2C$
by our definition.
\end{pf}
%
\begin{lemma} \label{lowerboundtau0} There exists a constant $c=c(\rho
)>0$ such that for any $i\geq1$ and $k\geq1$, we have
\[
\prob_0 ( \tau_0 \geq c k n_i
h_i) \geq\frac{c^k }{ h_i}
\]
and
\[
\prob_0 ( \tau_0 \geq c k n_i
h_i \mid\tau_{h_i} < \tau_0)
\geq{c^k}.
\]
\end{lemma}
\begin{pf} Starting from $h_i$, the probability of visiting $0$ before
returning to $h_i$ is $(3h_i)^{-1}$ by Lemma~\ref{easyfacts}, hence the
probability of having $kh_i$ visits to $h_i$ before hitting $0$ has
probability $(1-(3h_i)^{-1})^{kh_i}$. By Lemma~\ref{expander}, with
probability at least $\delta/3$ the random walk starting at $h_i$
spends at least $\delta n_i$ steps in the expander $E_{i}$, where
$\delta>0$ is the constant from Lemma~\ref{expander} ($h_i$ has degree
$3$). Given the number $K$ of visits to $h_i$ before returning to $0$,
the time spent away from $0$ is distributed as the sum of $K$ i.i.d.
random variables each being at least $\delta n_i$ with probability at
least $\delta/3$. We deduce that there exists some constant $c>0$ such
that for any $k\geq1$, we have
\[
\prob_0 ( \tau_0 \geq\tau_{h_i} + c k
n_i h_i \mid\tau_{h_i} < \tau_0
) \geq c^k.
\]
By Lemma~\ref{easyfacts} the event $\tau_{h_i} < \tau_0$ occurs with
probability $h_i^{-1}$, completing the proof.
\end{pf}

For the next step we define $\tau_0^{(m)}$ to be the $m$th return time
to $0$. That is, $\tau_0^{(1)} = \tau_0$ and for $m>1$
\[
\tau_0^{(m)} = \min\bigl\{ t > \tau_0^{(m-1)}\dvtx
X_t = 0 \bigr\}.
\]
It will also be convenient to define $\tau_0^{(0)}=0$.
%
\begin{lemma} \label{numexc1} There exist constants $C=C(\rho)>0$ and
$c=c(\rho)>0$ such that for any $i \geq1$ and any $k\geq1$ we have
\[
\prob_0 \bigl( \tau_0^{(Ckh_i)} <
kh_i n_i \bigr) \leq Ce^{-ck}.
\]
\end{lemma}
\begin{pf} Since
\[
\tau_0^{(Ckh_i)} = \sum_{m=1}^{Ckh_i}
\bigl(\tau_0^{(m)} - \tau_0^{(m-1)}
\bigr),
\]
we learn that $\tau_0^{(Ckh_i)}$ is a sum of $Ckh_i$ i.i.d. random
variables distributed as $\tau_0$. By Lemma~\ref{lowerboundtau0}, the
probability of each of these variables to be at least $n_ih_i$ is at
least $ch_i^{-1}$ for some small $c>0$. Large deviation for binomial
random variable immediately gives that for large enough $C>0$ we have
\[
\prob \Biggl( \sum_{m=1}^{Ckh_i} \bigl(
\tau_0^{(m)} - \tau_0^{(m-1)}\bigr) \leq
kh_i n_i \Biggr) \leq Ce^{-c_1k}
\]
for some constant $c_1>0$.
\end{pf}

The following lemma shows that the random walk on $G$ spends most of
its time inside the appropriate expander.
%
\begin{lemma} \label{inexpander} There exists a constant $C=C(\rho)>0$
such that for any integer $t$ satisfying $h_i n_i \leq t \leq h_i^2
n_i$ for some $i\geq1$ we have
\[
\P_0 ( X_t \in E_i ) \geq1-
Ch_i^{-2}.
\]
\end{lemma}
\begin{pf} For convenience we write $h$ for $h_i$, $n$ for $n_i$. Let
$t_0= t- 2h^{12}$ and define iteratively $t_{k+1}$ to be
\[
t_{k+1} = \min \{ \ell> t_k\dvtx  X_\ell= h \}
\]
for any integer $k \ge0$. For each $k\ge1$, the walk between times
$t_k$ and $t_{k+1}$ does an excursion, starting and ending at $h$. Call
such an excursion \textit{good} if $t_{k+1}-t_k > 2h^{12}$ and
$X_{t_k+1}\in E_i$, \textit{short} if $t_{k+1}-t_k < h^{11}$ and
\textit{bad} if it is neither good nor short. Let $K_{\mathrm{good}}$ be the
index of the first good excursion and $K_{\mathrm{bad}}$ be the index of
the first bad excursion and define the following events:
\begin{eqnarray*}
A_0 &=& \bigl\{t_1 > t-h^{12} \cap
X_t \notin E_i \bigr\},
\\
A_1 &=& \{K_{\mathrm{good}} \ge h \},
\\
A_2 &=& \{K_{\mathrm{bad}} \le h \}.
\end{eqnarray*}

We claim that if neither of these event occur, then $X_t\in E_i$.
Indeed, if $A_0$ did not occur then either $X_t \in E_i$ or $t_1 \leq
t-h^{12}$. In the latter case, since both $A_1$ and $A_2$ do not occur,
the first $K_{\mathrm{good}}-1 < h$ excursions are short and are followed
by a good excursion. The total length of these short excursions is no
more than $h^{12}$; hence $t_{K_{\mathrm{good}}} < t$ and the length of
the good excursion is at least $2 h^{12}$ so $t_{K_{\mathrm{good}}+1} >
t$. By definition, at all times in $(t_k,t_{k+1})$ of a good excursion
the walker is located in $E_i$, hence $X_t \in E_i$.

We now bound the probabilities of the three events, starting with
$A_1$. For any excursion, the probability that $X_{t_k+1}\in E_i$ is
$\frac13$. By Lemma~\ref{expander} we have that
\[
\P_0 \bigl(t_{k+1}-t_k \geq2 h^{12}
\mid X_{t_k+1}\in E_i \bigr) \geq c
\]
for some constant $c=c(\rho)>0$. Hence, the probability that an
excursion is good is at least $c/3$. Since the excursions are
independent, the probability of $A_1$ is bounded by $(1-c/3)^h = o(h^{-2})$.

We now bound $\P(A_2)$. An excursion is bad only if it is either longer
than $h^{11}$ and $X_{t_k+1} \notin E_i$, or its length is in
$[h^{11},2h^{12}]$ and $X_{t_k+1} \in E_i$. If $X_{t_k+1}=h+1$, then it
is standard that the probability that the walk does not return to $h$
in $h^{11}$ steps is of order $h^{-{11}/{2}}$. If $X_{t_k+1}=h-1$,
then the probability that the walk does not return to $h$ in $h^{11}$
steps is bounded by $e^{-c h^9}$ for some constant $c>0$. Indeed, there
are less than $2h$ edges below $h$ and the resistance from any vertex
below $h$ is at most $h$. The commute time identity now implies that
from any vertex below $h$ the probability of hitting $h$ within $4h^2$
steps is at least $1/2$, and the $e^{-ch^9}$ bound follows by iterating
this. Finally, if $X_{t_k+1}\in E_i$, then (\ref{mixing}) implies that
the the probability of $t_{k+1}-t_k \in[h^{11},2 h^{12}]$ is $O(h^{12}
n^{-1})$ which is $O(h^{-3})$ by (\ref{seqcond}). We get that the
probability that an excursion is bad is $O(h^{-3})$, hence the
probability that one of the first $h$ excursions is bad is $O(h^{-2})$.

We are left to bound $\P(A_0)$. If $X_{t_0+1}\in E_i$, then for $A_0$
to occur we must have that $t_1-t_0 \in[h^{12},2 h^{12}]$. As before,
(\ref{mixing}) implies that the probability of $t_1-t_0 \in[h^{12},2
h^{12}]$ is $O(h^{-3})$. If $X_{t_0+1}$ is below $h$, then the
probability of not hitting $h$ in the next $h^{12}$ steps is bounded by
$e^{-c h^{10}}$, by the same argument as above, using the commute time identity.

Finally, we need to bound the probability that $A_0$ occurs, and
$X_{t_0+1}$ is above $h$. Let $N_0$ be the number of visits to $0$ by
time $t$. By Lemma~\ref{numexc1} (with $k=h$) there are constants
$C,c>0$ such that
\[
\prob\bigl( N_0 \geq Ch^2 \bigr) \leq
Ce^{-ch}.
\]
In each such excursion from $0$ to $0$ the probability of reaching
$h^5$ is $h^{-5}$ by Lemma~\ref{easyfacts}. Hence,\vspace*{1pt} the probability that
the walk reaches height $h^5$ before time $t$ is at most
$C(h^{-3}+e^{-ch})$. Now, from any vertex between $h$ and $h^5$, the
expected time to hit either $h$ or $h^5$ is at most $h^{10}$.
Therefore, the probability that the walk does not hit $h$ or $h^5$ in
$h^{12}$ steps is at most $e^{-ch^2}$, for some constant $c>0$. Put
together, the probability that $X_{t_0+1}$ is above $h$, but $t_1>t_0 +
h^{12}$ is bounded by $C h^{-3}$.
%
\end{pf}
\begin{pf*}{Proof of Theorem~\ref{sharpness}}
Fix $i$ and abbreviate $t=t_i$, $h=h_i$ and $n=n_i$. We have that
\[
\prob_{0} (\tau_0 \geq t) = \prob_{0}(
\tau_0 \geq t \mbox{ and }\tau_0 < \tau_h) +
h^{-1} \prob_{0} (\tau_0 \geq t \mid
\tau_h < \tau_0).
\]
The first term\vspace*{1pt} is negligible since starting from any vertex $v$ between
$0$ and $h$, we have $\P_v(\tau_0 \wedge\tau_h \geq2h^2) \le\frac12$
by Lemma~\ref{easyfacts}, and hence, by the Markov property, $\P_0(\tau_0 \wedge\tau_h \ge t) \le e^{-c t / h^2}$. Theorem~\ref{mainthm1}
gives that $\P_0(\tau_0\ge t)\geq4^{-1} t^{-1/2}$ and since $t \geq
h^3$, we conclude that
%
\begin{equation}
\label{geqbd} \prob_{0} (\tau_0 \geq t) = \bigl(1+o(1)
\bigr) h^{-1} \prob_{0} (\tau_0 \geq t \mid
\tau_h < \tau_0).
\end{equation}

Assuming the event $\tau_h<\tau_0$ occurred, let $T_0= \tau_h$ and for
$j\geq1$ define
\[
T_j = \min\{ t > T_{j-1}\dvtx  X_t =h \},
\]
to be the time of the $j$th visit to $h$. Also, let $J=\max\{j\dvtx
T_j < \tau_0\}$ be the index of the last visit to $h$ before returning
to $0$. We define a sequence of random bits $\{b_j\}_{j\geq0}$ in the
following way. We set $b_j=1$ if $X_t = 0$ for some $T_j < t < T_{j+1}$
and $b_j=0$ otherwise. Conditioned on the history of the walk until
$T_j$ the probability of $b_{j}=1$ is exactly $(3h)^{-1}$, since the
walk needs to take a step to $h-1$ and then the probability of hitting
$0$ before $h$ is $h^{-1}$, by Lemma~\ref{easyfacts}. Hence, the
distribution of $J$ is geometric with parameter $(3h)^{-1}$.

Observe that the distribution of the walk between $T_J$ and $\tau_0$ is
that of a simple random walk started at $h$ and conditioned to hit $0$
before returning to $h$ and is independent of the walk until time $T_J$.
In particular, $T_J$ is independent of $\tau_0-T_J$. We may now bound
$\P_0(\tau_0=t)$ from below by
\[
\P_0(\tau_0=t)\ge h^{-1} \P_0
\bigl(T_J = t-(\tau_0 - T_J) \mid
\tau_h < \tau_0\bigr).
\]
Since $T_J$ is independent of $\tau_0-T_J$ we may condition on the
event $\tau_0-T_J=t-s$ and get that
%
\begin{eqnarray}
\label{sharpnesscondition} &&\P_0\bigl(T_J = t-(
\tau_0 - T_J) \mid\tau_h <
\tau_0\bigr)
\nonumber\\[-8pt]\\[-8pt]
&&\qquad=\sum_{s} \prob(\tau_0-T_J
= t-s \mid \tau_h < \tau_0) \prob(T_J=s
\mid\tau_h < \tau_0).
\nonumber
\end{eqnarray}
When starting a simple random walk at $h-1$, the expected hitting time
of $h$ or $0$ is at most $2h^2$ and the probability of hitting $0$
before $h$ is $h^{-1}$ by Lemma~\ref{easyfacts}. Thus, $\E[\tau_0 -
T_J \mid\tau_h < \tau_0] \leq2h^3$ hence
\[
\P_0\bigl( \tau_0 - T_J >
4h^3 \mid\tau_h < \tau_0 \bigr) \leq
\tfrac12.
\]
Therefore, it is enough to show that for any $s$ satisfying $t-4h^3\le
s \le t$, we have
%
\begin{equation}
\label{sharpnesstoshow} \P_0(T_J = s \mid
\tau_h < \tau_0) \ge\frac{c
}{ hn} \P_0(
\tau_0 \geq t \mid\tau_h < \tau_0),
\end{equation}
since then by (\ref{geqbd}) and (\ref{sharpnesscondition}) we get that
\[
\P_0(\tau_0=t) \geq\frac{ c }{ h n}
\P_0(\tau_0 \geq t) = \Theta \biggl( \frac{\log t }{ t}
\biggr) \P_0(\tau_0 \geq t).
\]

To show (\ref{sharpnesstoshow}) we take some small $\delta$ and bound
\[
\P_0(T_J = s \mid\tau_h <
\tau_0) \geq\prob_0(T_J=s,
X_{s-\delta n} \in E_i, \tau_0 \geq s \mid
\tau_h < \tau_0).
\]
By the Markov property the last probability is at least
%
\begin{equation}
\label{sharpnessmidstep2}\qquad\P_h(\tau_0 <
\tau_h) \min_{u \in E_i} \P_u(\tau_h=
\delta n) \P_0(X_{s-\delta n} \in E_i,
\tau_0 \geq s-\delta n \mid\tau_h < \tau_0).
\end{equation}
Lemmas~\ref{outofexpander} and~\ref{easyfacts} give that the
product of the first two probabilities is at least $c(hn)^{-1}$. By
Lemmas~\ref{inexpander} and~\ref{easyfacts} we have that
\[
\P_0(X_{s-\delta n} \notin E_i\mid
\tau_h < \tau_0) \le C h^{-1},
\]
and since $s=(1+o(1))chn\log n$, Lemma~\ref{lowerboundtau0} with
$k=c\log n$ gives that
\[
\P_0 (\tau_0 \geq s- \delta n\mid\tau_h <
\tau_0) \geq\delta^{c
\log
n} \geq h^{-0.5}
\]
as long as $c>0$ is chosen to be small enough. Since $\mu(A \cap B)
\geq\mu(B) - \mu(A^c)$ for any probability measure $\mu$ and events
$A,B$ we get that
\begin{eqnarray*}
\P_0 (X_{s-\delta n} \in E_i, \tau_0
\geq s-\delta n \mid\tau_h < \tau_0) &\geq& \bigl(1-o(1)
\bigr) \P_0 (\tau_0 \geq s-\delta n \mid
\tau_h < \tau_0)
\\
&\geq& \bigl(1-o(1)\bigr)\P_0(\tau_0 \geq t \mid
\tau_h < \tau_0 ),
\end{eqnarray*}
which together with (\ref{sharpnessmidstep2}) shows (\ref
{sharpnesstoshow}) and the proof is complete.
\end{pf*}

\section{Combs}\label{combssec}

Recall the definition of the comb product of two graphs and of the
finite collision property in Section~\ref{finitecolsec}. In this
section we prove that the graph $G=G(\{h_i,n_i\})$, for $\{h_i, n_i\}$
satisfying (\ref{seqcond}), is such that $\comb_0(\Z,G)$ does not have
the finite collision property. We begin with a sketch to illustrate the
idea of the proof.

In the rest of this section we sometimes write $h,n$ for $h_i,n_i$,
respectively. Our goal is to get the two walkers inside the same
expander $E_i$ since then they collide with positive probability by
Lemma~\ref{collisioninsideexpander}. Starting from the base of the
comb, the probability of reaching height $h$ before returning to the
base is $(3h)^{-1}$. If this happens, the random walk has positive
probability of being ``swallowed'' in the expander $E_i$ and staying in
it for $n$ steps. At each visit to the tip of the expander, that is,
the vertex $h$, the probability of getting back to the base of the comb
before returning to $h$ is $(3h)^{-1}$. The other expanders, above and
below $E_i$, are either too small or too far away to matter. We deduce
that by time $hn$ the typical behavior of the walker is to walk about
$h$ steps on the base of the comb, then rise to height $h$, have about
$h$ excursions of length $n$ inside the expander and finally return to
the base of the comb.

Thus, after $h^2 n$ steps, each random walker has performed about $h^2$
steps on the base of the comb (this is a simple random walk on $\Z$)
and in about $h$ of them it performs excursions of length $hn$ in which
it spends most of the time in the expander~$E_i$. The base points on
$\Z
$ of these $h$ long excursions are roughly $h$ uniform points in $\{-h,\ldots, h\}$, so the probability that in at least one of them the two
walkers are over the same base point is uniformly positive. We conclude
that by time $h^2 n$ the two walkers have positive probability of
colliding. This occurs in all scales, that is, for all $i\geq0$. Each
scale has almost no influence on what occurs in the next scale; hence
we get the required result.

We now make this heuristic precise. Given a simple random walk $X_t$ on
$\comb(\Z,G)$ we write $X_t^{(1)}$ and $X_t^{(2)}$ for its first and
second coordinates, respectively. Note that $X_t^{(1)}$ is a time change
of a simple random walk on $\Z$, and $X_t^{(2)}$ is distributed
precisely as simple random walk on $G(\{h_i,n_i\})$ equipped with extra
two loops at $0$. One can easily check that the estimates of Section
\ref{secsharp} are valid for this graph as well. Put $T_0=0$ and $T_i =
T_{i-1} + n_i h_i^2$. For any $i\geq1$ and $k=1,\ldots, h_i$ define
the events
\[
I_k = \bigl\{ X_{T_{i-1} + kh_in_i}^{(1)} =
Y_{T_{i-1} + kh_in_i}^{(1)} \mbox{ and }X_{T_{i-1} + kh_in_i}^{(2)}\in
E_i \mbox{ and } Y_{T_{i-1} +
kh_in_i}^{(2)} \in
E_i \bigr\}.
\]
The following lemma is the key step for proving Theorem \ref
{nocollisionthm}. We remark that all constants in this section depend
on $\rho$ from the definition of $G$ in the previous section.
%
\begin{lemma}\label{momarg} There exists a constant $c=c(\rho)>0$ such
that for all $i \geq1$ we have
\[
\prob \Biggl( \bigcup_{k=1}^{h_i}
I_k \Bigm| X_{T_{i-1}}, Y_{T_{i-1}} \Biggr) \geq c.
\]
\end{lemma}
\begin{pf*}{Proof of Theorem~\ref{nocollisionthm}} Write $\A_i$ for
the event that $X_t$ and $Y_t$ collide in the time interval $[T_{i-1},
T_i]$. Lemma~\ref{momarg} together with Lemma \ref
{collisioninsideexpander} shows that
\[
\prob(\A_i \mid X_{T_{i-1}}, Y_{T_{i-1}} )\geq c
\]
for some constant $c>0$. We deduce that $\A_i$ occurs infinitely many
times with probability $1$, completing the proof.
\end{pf*}
%

We will prove Lemma~\ref{momarg} using a second moment argument;
however, we require two additional preparatory lemmas about the random
walk on a single copy of $G(\{h_i,n_i\})$, the tooth of the comb.
%
\begin{lemma}\label{returnfrommiddle} Consider the simple random walk
on $G$. There exist constants $C=C(\rho)>0$ and $c=c(\rho)>0$ such that
for any $i$ and any vertex $v$ satisfying $0 \leq h(v)\leq h_i^4 $, we have
\[
\E_v \bigl[ e^{c n_i^{-1} (\tau_0\wedge\tau_{h_i} \wedge\tau
_{h_i^4})} \bigr] \leq C.
\]
\end{lemma}
\begin{pf} For any vertex $v$ of $G$ with height between $0$ and $h^4$
we have that $\E_v ( \tau_0 \wedge\tau_h \wedge\tau_{h^4}) \leq Cn$.
To see this observe that there are three cases: if $v$ is in the
expander $E_{h}$ the expected hitting time of $h$ is $O(n)$ by the
commute time identity and (\ref{reff}). If $h(v)>h$, then the expected
time to hit either $h$ or $h^{4}$ is, by the commute time identity, at
most $h^8$, which is $o(n)$ [there are no expanders between $h$ and
$h^4$ by (\ref{seqcond})]. Similarly, if $h(v) < h$, then the expected
time to hit either $0$ or $h$ is $o(n)$. Hence for any such $v$ we have
\[
\prob_v (\tau_0 \wedge\tau_h \wedge
\tau_{h^4} \geq2Cn ) \leq\tfrac{1 }{2},
\]
hence
\[
\prob_v (\tau_0 \wedge\tau_h \wedge
\tau_{h^4} \geq Bn ) \leq e^{-cB},
\]
and the (ii) follows by integration.\vadjust{\goodbreak}
\end{pf}
%
\begin{lemma} \label{returnestimates} Consider the simple random walk
on $G$. There exist constants $C=C(\rho)>0$ and $c=c(\rho)>0$ such that
for any $i$ and any $B>0$,
\[
\prob_0( \tau_0 \wedge\tau_{h_i^4} \geq
Bn_i h_i ) \leq\frac{2 e^{-cB}
}{ h_i},
\]
hence
\[
\E_0 e^{c (n_i h_i)^{-1} (\tau_0 \wedge\tau_{h_i^4})} \leq1 + \frac{C
}{ h_i}.
\]
%
\end{lemma}
\begin{pf}
%
Let $N_h$ denote the number of visits to $h$ before time $\tau_0
\wedge
\tau_{h^4}$. We have that
\[
\prob_0 ( N_h \geq k) \leq\frac{1 }{ h} \biggl( 1
- \frac{1 }{3h} \biggr)^{k-1} \leq\frac{e^{-(k-1)/3h} }{ h},
\]
since reaching to $h$ before $0$ has probability $h^{-1}$, and given
that at each visit to~$h$, the probability of visiting $0$ before
returning to $h$ is precisely $(3h)^{-1}$. By this bound it suffices to
prove that
%
\begin{equation}
\label{midpoint} \prob_0 ( \tau_0 \wedge
\tau_{h^4} \geq Bnh \mbox{ and }N_h \leq cBh ) \leq
\frac{e^{-cB} }{ h}
\end{equation}
for some small $c>0$. Let $\gamma_m$ for $m=1,\ldots, cBh$ be i.i.d.
random variables distributed as the stopping time $\tau_0 \wedge\tau_h
\wedge\tau_{h^4}$ for the random walk starting at $h$. Then the
probability on the left-hand side of (\ref{midpoint}) is at most the
probability that $\sum_{m=1}^{cBh} \gamma_m \geq Bnh$. By Lemma \ref
{returnfrommiddle} we have that there exists some $c_2>0$ such that
$\E_0 e^{c_2 n^{-1} \gamma_m} \leq C$. Hence, by independence and
Markov's inequality we get that
\[
\prob \Biggl(\sum_{m=1}^{cBh}
\gamma_m \geq Bnh \Biggr) \leq\frac{C^{c Bh}
}{ e^{c_2 Bh}},
\]
which is of order $e^{-cBh}$ if $c=c(c_2,C)>0$ is chosen small enough
compared with~$c_2$. This proves (a stronger assertion than) (\ref
{midpoint}) and concludes the proof.
%
\end{pf}

Consider now the random walk $X_t$ on $\comb(\Z, G)$. Write $X^{(2)}_t$
for the second coordinate of $X_t$ and for any $i$ let $\ell_i(t)$
denote the random variable
\[
\ell_i(t) = \bigl| \bigl\{ j \in[T_{i-1}, T_{i-1}+t]\dvtx
X^{(2)}_{j-1} = X^{(2)}_{j} = 0 \bigr\}
\bigr|.
\]
In other words, $\ell_i(t)$ counts the number of times $j \in[T_{i-1},
T_{i-1}+t]$ in which $X_j$ walked on the $\Z$ base of the comb.
%
\begin{lemma}\label{numexcursion} Consider the simple random walk on
$\comb(\Z, G)$. There exist constants $C=C(\rho)>0$ and $c=c(\rho)>0$
such that for any $i$ and any $k=1,\ldots, h_i$ we have
\[
\prob \bigl( \ell_i(kh_i n_i) \geq
Ckh_i \bigr) \leq C e^{-ck}\vadjust{\goodbreak}
\]
and
\[
\prob \bigl( \ell_i(kh_in_i) \leq
C^{-1} k h_i \bigr) \leq\frac{1 }{
h_i^2} + C
e^{-ck}.
\]
\end{lemma}
\begin{pf} Part (i) of the lemma is equivalent to Lemma~\ref{numexc1}.
For $m\geq1$ write $t_m$ for the time in which $X_t$ takes the $m$th
step on $\Z$. That is, $t_0=0$ and for $m\geq1$ we have
\[
t_m = \min \bigl\{ j > t_{m-1}\dvtx  X^{(2)}_{j-1}
= X^{(2)}_{j} = 0 \bigr\}.
\]
To prove the second assertion of the lemma, note that the event $\ell_i(khn) \leq C^{-1} k h$ is equivalent to
%
\begin{equation}
\label{badevent} \sum_{m=1}^{C^{-1}kh}
(t_m - t_{m-1}) \geq khn.
\end{equation}
For each $m$, let $A_m$ be the event that $X_t$ visited the vertex
$h^4$ on one of the comb's teeth between times $t_{m-1}$ and $t_m$, and
write $\bar{A}_m$ for the complement of $A_m$. By Lemma~\ref{easyfacts}
we have that $\prob(A_m)=h^{-4}/3$. Thus, the probability that $A_m$
occurs for some $m=1,\ldots, C^{-1}kh$ is at most $h^{-2}$ since
$k\leq h$. We get that
\[
\prob \Biggl( \sum_{m=1}^{C^{-1}kh}
(t_m - t_{m-1}) \geq khn \Biggr) \leq\frac{1 }{ h^2} +
\prob \Biggl( \sum_{m=1}^{C^{-1}kh}
(t_m - t_{m-1}){\mathbf1}_{\bar{A}_m} \geq khn \Biggr).
\]
To bound the last term of this inequality, observe that
\[
(t_m-t_{m-1}){\mathbf1}_{\bar{A}_m} \stackrel{(d)} {\leq}
\tau_0 \wedge \tau_{h^4},
\]
where $\tau_0$ and $\tau_{h^4}$ are the corresponding hitting times on
$G$. By Lemma~\ref{returnestimates} there exists some $C_2>0$ such that
\[
\E e^{c (nh)^{-1} (t_m-t_{m-1}){\mathbf1}_{\bar{A}_m}} \leq1 + \frac{C_2
}{
h},
\]
and by independence and Markov's inequality we deduce that
\[
\prob \Biggl( \sum_{m=1}^{C^{-1}kh}
(t_m - t_{m-1}){\mathbf1}_{\bar{I}_m} \geq khn \Biggr)
\leq\frac{  ( 1+ {C_2 }/{ h}  )^{C^{-1} kh} }{
e^{ck} },
\]
which is at most $Ce^{-ck}$ if $C=C(c,C_2)>0$ is chosen large enough.
This completes the proof.
\end{pf}
%
\begin{lemma} \label{firstmoment} There exist constants $C=C(\rho)>0$
and $c=c(\rho)>0$ such that for any $i$ and any $k=1,\ldots, h_i$
\[
\frac{c }{\sqrt{kh_i}} \leq\P( I_k \mid X_{T_{i-1}},
Y_{T_{i-1}}) \leq\frac{C }{\sqrt{kh_i}}.
\]
\end{lemma}
\begin{pf}
Lemma~\ref{numexcursion} implies that for some positive constants $C,c$
we have
\[
\prob \bigl( C^{-1} kh \leq\ell(khn) \leq Ckh \mid X_{T_{i-1}},
Y_{T_{i-1}} \bigr) \geq1-Ce^{-ch}-Ch^{-2}.
\]
So with this probability, this holds for both walks $X_t$ and $Y_t$.
Clearly $X_{T_{i-1}}^{(1)}$ and $Y_{T_{i-1}}^{(1)}$ are of distance at
most $T_{i-1}$ away from the origin of $\comb(\Z,G)$, and $T_{i-1}\ll
\sqrt{h}$ by (\ref{seqcond}). Thus, the local CLT for the simple random
walk on $\Z$ implies that the probability that at time $T_{i-1}+khn$
the two walkers are in the same copy of $G$ is at least $c(kh)^{-1/2}$
and at most $C(kh)^{-1/2}$. This shows $\prob(I_k \mid X_{T_{i-1}},
Y_{T_{i-1}})\leq C(kh)^{-1/2}$. Furthermore, by Lemma~\ref{inexpander}
the probability that at that time the walks are not inside the expander
$E_{h}$ is at most $Ch^{-2}$. The lower bound $\prob(I_k \mid
X_{T_{i-1}}, Y_{T_{i-1}})\geq c(kh)^{-1/2}$ follows.
\end{pf}
%
\begin{lemma}\label{secondmoment} There exists a constant $C=C(\rho)>0$
such that for any $i$ and any $k_1 < k_2$ in $\{1,\ldots, h_i\}$, we have
\[
\prob(I_{k_1} I_{k_2} \mid X_{T_{i-1}},
Y_{T_{i-1}}) \leq\frac{C
}{ h_i \sqrt{k_1(k_2-k_1)}}.
\]
\end{lemma}
\begin{pf}
As in the proof of Lemma~\ref{firstmoment}, with probability at least
$1-O(h^{-2})$, we have that $\ell(k_1 h n) = \Theta(k_1 h)$ and $\ell
(k_2 h n) = \Theta(k_2 h)$. Also as before, since $T_{i-1} \ll\sqrt {h}$ the local CLT for the simple random walk on $\Z$ implies that the
probability that both at times $T_{i-1}+k_1hn$ and $T_{i-1} + k_2hn$
the two walkers are in the same copy of $G$ is $O( h^{-1}
(k_1(k_2-k_1))^{-1/2})$. Lemma~\ref{inexpander} shows that the
probability that in each of these occurrences the walkers are not in
the respective expanders $E_h$ is $O(h^{-2})$, and the lemma follows.
\end{pf}
\begin{pf*}{Proof of Lemma~\ref{momarg}} Lemma~\ref{firstmoment}
gives that
\[
\sum_{k=1}^{h} \prob(I_k \mid
X_{T_{i-1}}, Y_{T_{i-1}}) \geq c,
\]
and Lemma~\ref{secondmoment} yields that
\[
\sum_{k_1=1}^h \sum
_{k_2=1}^h \prob(I_{k_1} I_{k_2}
\mid X_{T_{i-1}}, Y_{T_{i-1}}) \leq C.
\]
The lemma follows by the inequality $\prob(X>0)\geq(\E X)^2/\E X^2$
valid for any nonnegative random variable $X$.
\end{pf*}

\section*{Acknowledgments}

The authors thank Omer Angel, Noam
Berger, Gady Kozma, Russ Lyons, Elchanan Mossel, Yuval Peres and Perla
Sousi for useful discussions. We thank the anonymous referee for many
useful comments and suggestions.


%

\printaddresses

\end{document}